\documentclass{article}
\usepackage[utf8]{inputenc}
\usepackage{geometry}
\usepackage{graphicx}
\graphicspath{ {./images/} }
\usepackage{amsmath} % Equations
\usepackage{amssymb} % Equations
\usepackage{listings}
\usepackage{color}
\usepackage[vcentermath]{youngtab}
\usepackage{tikz}
\usepackage{lscape}
\usepackage{amsthm}
\usepackage{hyperref}
\usepackage{authblk}
\usepackage[linesnumbered,lined,commentsnumbered,ruled]{algorithm2e}

\theoremstyle{definition}

\newtheorem{theorem}{Theorem}

\newtheorem{assumption}{Assumption}

\newtheorem{definition}{Definition}

\newtheorem{lemma}{Lemma}

\newtheorem{example}{Example}
\newtheorem{remark}{Remark}

\title{Epsilon local rigidity and numerical algebraic geometry}
\author[1]{Andrew Frohmader}
\author[2]{Alexander Heaton}
\affil[1]{University of Wisconsin, Milwaukee}
\affil[2]{Max Planck Institute for Mathematics in the Sciences, Leipzig, and Technische Universit\"at Berlin}
\date{February 12, 2020}

\begin{document}

\maketitle

\begin{abstract}
A well-known combinatorial algorithm can decide generic rigidity in the plane by determining if the graph is of Pollaczek-Geiringer-Laman type. Methods from matroid theory have been used to prove other interesting results, again under the assumption of generic configurations. However, configurations arising in applications may not be generic. We present Theorem \ref{theorem:epsilon-local-rigidity} and its corresponding Algorithm \ref{algorithm:epsilon-local-rigidity} which decide if a configuration is $\varepsilon$-locally rigid, a notion we define. A configuration which is $\varepsilon$-locally rigid may be locally rigid or flexible, but any continuous deformations remain within a sphere of radius $\varepsilon$ in configuration space. Deciding $\varepsilon$-local rigidity is possible for configurations which are smooth or singular, generic or non-generic. We also present Algorithms \ref{algorithm:discrete-flex} and \ref{algorithm:infinitesimal-flex-parameter-homotopy} which use numerical algebraic geometry to compute a discrete-time sample of a continuous flex, providing useful visual information for the scientist.
\end{abstract}

\section{Introduction}\label{section:introduction}

Consider a graph with $n$ nodes labelled by $[n]=\{1,2,\dots,n\}$, and with edge set $E \subset \binom{[n]}{2}$. A map called the initial \textit{configuration} $p_0:[n] \to \mathbb{R}^d$ embeds this graph in some $\mathbb{R}^d$. %We embed this graph in some $\mathbb{R}^d$ by a map called the initial \textit{configuration} $p_0:[n] \to \mathbb{R}^d$.
We make precise definitions in Section \ref{section:preliminaries}, but the basic idea is to consider the edge distances $\ell_{ij}$ between nodes connected by an edge, but also the distances $\widehat{\ell_{ij}}$ between nodes that are not connected by an edge. If the nodes can move so that the distances $\ell_{ij}$ for $\{i,j\} \in E$ remain constant, but some distance $\widehat{\ell_{ij}}$ for $\{i,j\} \notin E$ changes, then we say the configuration $p_0$ is \textit{flexible}. If no such continuous motion exists, we say the graph is \textit{locally rigid}. We can always translate or rotate the graph inside $\mathbb{R}^d$, called a rigid motion, which keeps all pairwise distances constant. Thus, to decide local rigidity, we need to establish if there are deformations of the embedded graph \textit{besides the rigid motions}.

It is often difficult to determine conclusively whether a given configuration $p_0$ admits deformations which preserve edge-lengths, yet are not rigid motions (see \cite{A2009} where they show this problem is coNP-hard). Such deformations preserving the $\ell_{ij}$ but changing some non-edge length $\widehat{\ell_{ij}}$ are called \textit{flexes}.
We demonstrate these ideas with an historically important example (appearing in the first 3 competing patents for tensegrity \cite{JuanTur2008TensegrityFrameworksStaticReview}) which we call the 3-prism, whose rigidity is well-known. Any generic configuration of the 3-prism is infinitesimally rigid, and therefore locally rigid (see the well-known Theorem \ref{thm:FundamentalTheorem1} below).
But there are many non-generic configurations which lie on the singular locus of the polynomial system of member constraints. The left side of Figure \ref{fig:possible-flex-3prism} illustrates a sequence of configurations nearby such a singular $p_0$ appearing to deform the 3-prism.
These configurations were computed numerically using Algorithm \ref{algorithm:infinitesimal-flex-parameter-homotopy} below. Indeed, up to numerical tolerance, all the distances $\ell_{ij}$ between nodes connected by an edge are the same, yet several non-edge distances $\widehat{\ell_{ij}}$ change.
\begin{figure}[!htb]
    \centering
    \includegraphics[width=0.45\textwidth]{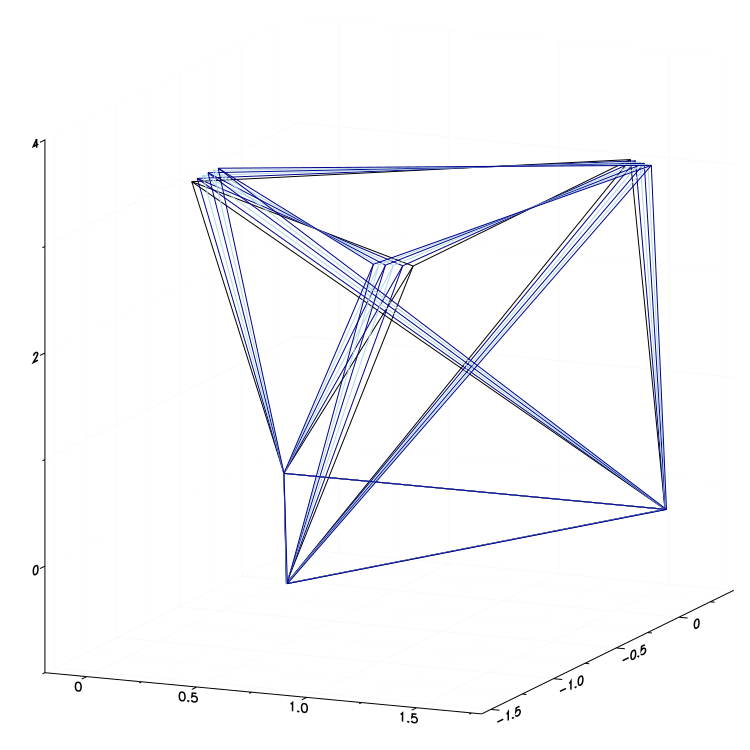} \includegraphics[width=0.45\textwidth]{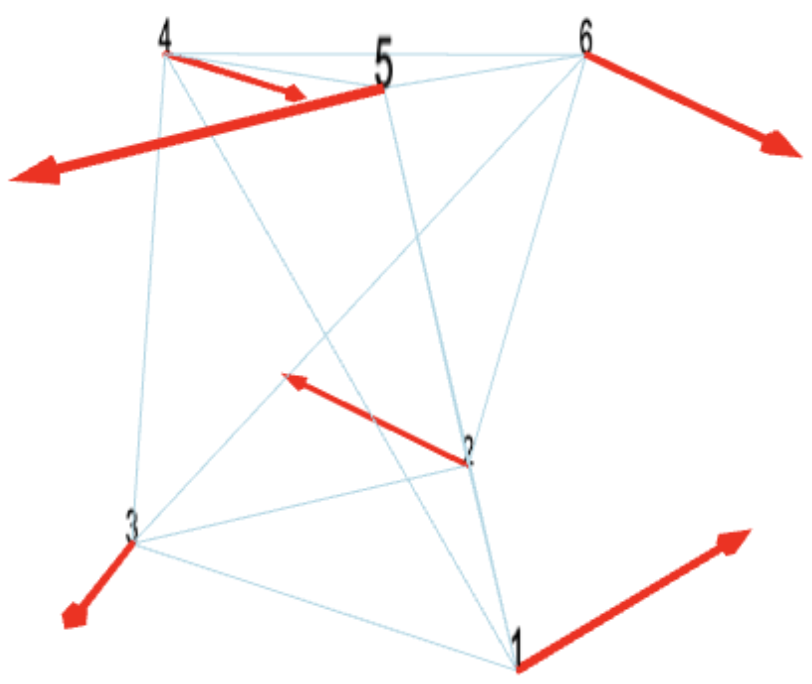}
    \caption{Numerical flex of the 3-prism}
    \label{fig:possible-flex-3prism}
\end{figure}
Theorem \ref{thm:FundamentalTheorem1} fails to apply since $p_0$ is singular and infinitesimal flexes exist (right side of Figure \ref{fig:possible-flex-3prism}). This infinitesimal flex was used as input to our Algorithm \ref{algorithm:infinitesimal-flex-parameter-homotopy}.
Given only this numerical evidence we may be tempted to think that $p_0$ is \textit{flexible}.
However, using our Theorem \ref{theorem:epsilon-local-rigidity} and Algorithm \ref{algorithm:epsilon-local-rigidity}, we can conclusively determine that the numerical configurations illustrated in Figure \ref{fig:possible-flex-3prism} are not part of a legitimate flex.
To be clear, our results do not prove local rigidity, but they prove something that may be good enough in practice, $\varepsilon$-local rigidity. In contrast to many methods and rules applied in rigidity theory, the methods presented here apply equally well to any configuration, be it exceptional or generic, singular or smooth.

Hence, the main results of this article are Theorem \ref{theorem:epsilon-local-rigidity} and its associated Algorithm \ref{algorithm:epsilon-local-rigidity}. These concern $\varepsilon$-local rigidity, which we depict in Figure \ref{fig:three-sphere-with-text}.
We make precise definitions in Section \ref{section:epsilon-local-rigidity}, but briefly, a configuration $p_0$ is $\varepsilon$-locally rigid if we can certify that any continuous flexes are extremely small, and \textit{cannot go far}.
Namely, we will certify that any continuous flex through $p_0$ stays within an $\varepsilon$-ball about $p_0$ within configuration space $\mathbb{R}^{nd}$.
\begin{figure}[!htb]
    \centering
    \includegraphics[width=0.9\textwidth]{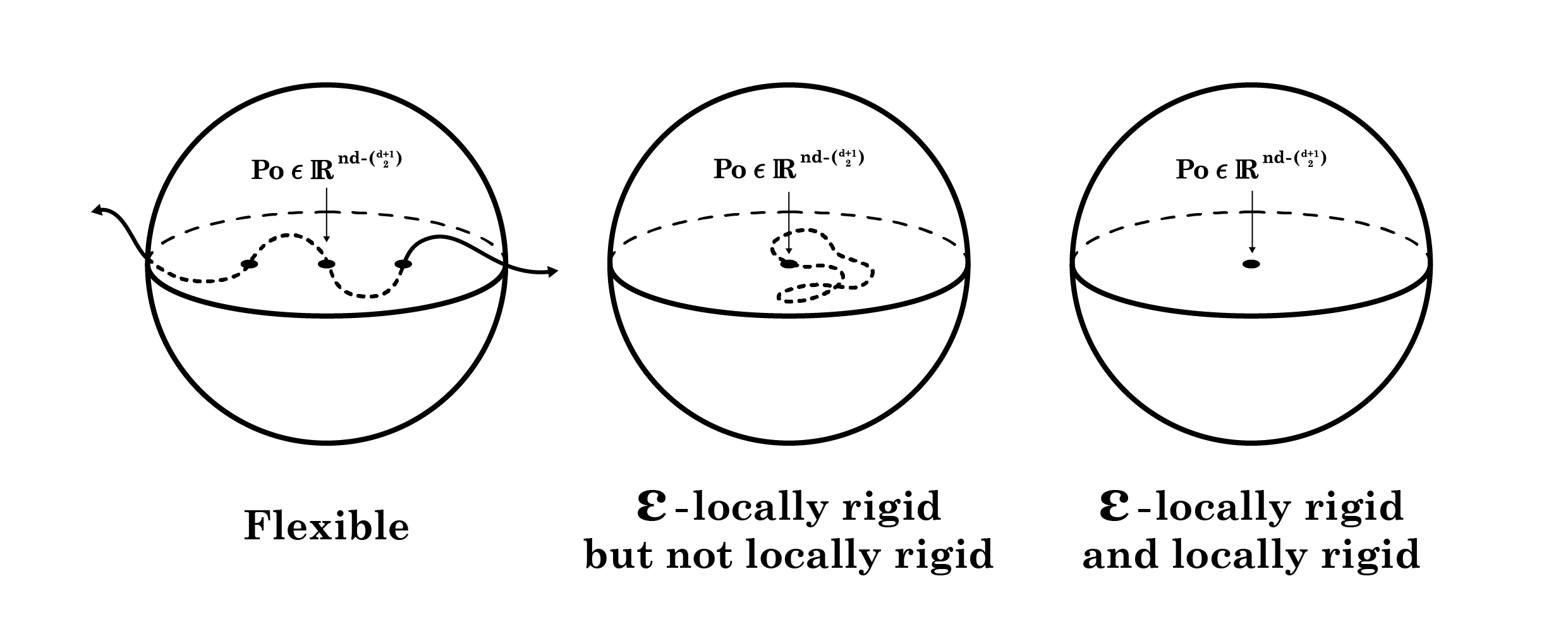}
    \caption{}
    \label{fig:three-sphere-with-text}
\end{figure}
Since the choice of $\varepsilon > 0$ is up to the user, knowing $\varepsilon$-local rigidity is practically as good as knowing local rigidity. Our methods use polynomial homotopy continuation and numerical algebraic geometry \cite{Wampler2013, SommeseWamplerTEXT}, taking advantage of theorems in real algebraic geometry \cite{AubryRouillierSafeyElDin2002, Hauenstein2012,  RouillierRoySafeyElDin2000, Seidenberg1954} which can guarantee finding real solutions to systems of polynomial equations, should they exist. In this way, our methods deal directly with the real algebraic set, applying equally well to both smooth and singular configurations. In many cases $\varepsilon$-local rigidity may be more relevant than local rigidity, since allowing some small movement may be acceptable for certain applications. Finally, our methods also imply Algorithms \ref{algorithm:discrete-flex} and \ref{algorithm:infinitesimal-flex-parameter-homotopy} which produce animations of a flex, should it exist, yielding easily-understandable visual information for the scientist. \\

\noindent \textbf{Related literature:} We briefly discuss some research related to this article. Since the simplest models are often the most useful, variations on this theme have been well-studied and therefore have many names: \textit{configuration, truss, bar-and-joint framework, tensegrity, linkage, assembly mode, structure, mechanism, mobility, degrees of freedom}, and more. The combinatorics community has studied this problem by assuming certain \textit{generic} conditions on $p_0$. This allows for theorems that use the graph structure alone (for overviews see \cite{CG2020, W1992} or Chapter 61 of \cite{GOT2018}), but requires generic assumptions. In contrast, the results of this paper apply to any configuration. The engineering and kinematics communities have also studied variations of this problem. In \cite{G2005}, a critical review of methods of mobility analysis is presented. The author enumerates 35 different approaches to calculating the mobility of a given configuration $p_0$, describing the limitations and outright failure of the methods in various cases. Again, in contrast, the results of this article apply to any configuration without assumptions. However, we cannot prove local rigidity, only $\varepsilon$-local rigidity.

The 2019 paper \cite{almostrigidity2019} defines the \textit{almost rigidity of frameworks}. This paper extracts information from the singular value decomposition of the rigidity matrix for a given configuration $p_0$, using existing tests for \textit{pre-stress stability} \cite{secondorderrigidity1980, higherorderrigidity1994} and semidefinite programming. Given a specific configuration $p_0$, if a certain list of conditions is met, they provide a specific radius $\eta_1 > 0$ of a sphere within which all continuous deformations of $p_0$ will remain. The size of $\eta_1$ depends on $p_0$. The key difference between their result and ours is that our methods allow a freely chosen $\varepsilon$ which can be adjusted according to the application. In contrast, the methods of \cite{almostrigidity2019} output a single radius $\eta_1$ computed from the SVD of the rigidity matrix for $p_0$.

\section{Preliminaries}\label{section:preliminaries}

This section collects the basic definitions and results from rigidity theory and numerical algebraic geometry needed for this article. For general references on combinatorial rigidity theory see \cite{CG2020, W1992} or Chapter 61 of \cite{GOT2018}. For numerical algebraic geometry see \cite{Wampler2013, SommeseWamplerTEXT}. We consider a connected graph $(V, E)$ with $n$ nodes and $m$ edges where $V = [n] =\{1,2,\dots,n\}$ and an edge is written either $\{i,j\} \in E$, or briefly $ij \in E$, where $i \in [n]$ and $j \in [n]$ correspond to nodes. The graph is embedded in $\mathbb{R}^d$ by the map $p_0:[n] \to \mathbb{R}^d$, called the \textit{initial configuration}. We assume that the affine span of the $n$ nodes is $d$ dimensional (they are not all contained in some lower-dimensional subspace). By slight abuse of notation we specify $p_0$ by a list of the $nd$ coordinates of its nodes, denoted by a tuple $(p_{ik}) \in \mathbb{R}^{nd}$, where $i \in [n]$ and $k \in [d]$, so that $p_{ik}$ is the $k$th coordinate of the $i$th node. A \textit{deformation} of $p_0$ is a continuous path $p(t):[0,1] \to \mathbb{R}^{nd}$ which recovers $p_0$ at $t=0$. A \textit{rigid motion} is a deformation which preserves the pairwise squared distances
\begin{equation}\label{equation:rigid-motion-squared-distances}
    \left\{ \sum_{k=1}^d (p_{ik}(t) - p_{jk}(t) )^2 \right\}_{ij \in \binom{[n]}{2}},
\end{equation}
at every time $t$, where $ij \in \binom{[n]}{2}$ runs over all pairs of nodes, not just the nodes that are connected by an edge $ij \in E$. We associate a system of polynomial equations called the \textit{member constraints} to our embedded graph which enforce the requirement that edges must have constant length.

\begin{definition}\label{definition:member-constraints}
Let $g: \mathbb{C}^{nd} \to \mathbb{C}^m$ be given by
  \begin{align*}
    g(x) &= \begin{bmatrix}
           g_{1}(x) \\
           g_{2}(x) \\
           \vdots \\
           g_{m}(x)
         \end{bmatrix},
  \end{align*}
where $g_l: \mathbb{C}^{nd} \to \mathbb{C}$ gives the difference in squared length of the $l^{th}$ edge from that of the initial configuration $p_0$. If the $l^{th}$ edge is $ij$, then
\begin{equation}\label{equation:member-constraints}
    g_l(x) = \sum_{k=1}^d (x_{ik} - x_{jk})^2 - \sum_{k=1}^d (p_{ik}-p_{jk})^2.
\end{equation}
\noindent The system of \textit{member constraints} associated to $p_0$ is the system of polynomial equations given by $g(x) = 0$. We will also need the corresponding algebraic set $V(g) := \{ x \in \mathbb{C}^{nd} \, : \, g(x) = 0 \}$ and real algebraic set $V_{\mathbb{R}}(g) := V(g) \cap \mathbb{R}^{nd}$.
\end{definition}

The real algebraic set $V_{\mathbb{R}}(g)$ is the \textit{configuration space}. We can now state the key definition.

\begin{definition}\label{definition:flex}\label{definition:local-rigidity}
A \textit{flex} of $p_0$ is a deformation $p(t):[0,1] \to \mathbb{R}^{nd}$ such that $g(p(t)) = 0$ for all $t \in [0,1]$ and which is not a rigid motion. The configuration $p_0$ is called \textit{locally rigid} if no flex exists.
\end{definition}

The first and easiest way to decide whether a given configuration $p_0$ is locally rigid is to examine the linearization of the polynomial map $g$ of Definition \ref{definition:member-constraints}. The Jacobian of $g$, denoted $dg$, is an $m \times nd$ matrix of polynomials. The matrix $dg$ is usually called the rigidity matrix. When we evaluate $dg$ at the point $p_0 \in V_\mathbb{R}(g)$, the vectors in its right nullspace $Null(dg|_{p_0})$ are sometimes called \textit{infinitesimal mechanisms}, for example in \cite{S2007}. The Lie algebra of the Euclidean group of rigid motions produces $\binom{d+1}{2}$ linearly independent infinitesimal mechanisms, whose span we denote $RM$. Any remaining vectors outside the span of the infinitesimal rigid motions are called \textit{infinitesimal flexes}, giving rise to the following decomposition:
\begin{equation}\label{equation:decomposition-of-nullspace}
    Null(dg|_{p_0}) = RM \oplus F,
\end{equation}
where any $v \notin RM$ is an infinitesimal flex. Any configuration $p_0$ for which $F=0$ is called \textit{infinitesimally rigid}. The good news comes in the form of the well-known theorem
\begin{theorem}[\cite{AsimowRoth1979}]\label{thm:FundamentalTheorem1}
Infinitesimal rigidity implies local rigidity.
\end{theorem}

Thus if $\text{rank}\big( dg|_{p_0}\big) = nd - \binom{d+1}{2}$ then our configuration is \textit{infinitesimally rigid}, and therefore also \textit{locally rigid}. Crucially, the converse is false - a configuration may have infinitesimal flexes but still be locally rigid. Other techniques are required to determine if infinitesimal flexes are actually realizable. At most points $p \in V(g)$ the rank of the Jacobian will be equal to the \textit{generic rank} of the polynomial map $g$ \cite[Section 13.4]{SommeseWamplerTEXT}. But at \textit{singular points} $p_0$ the Jacobian $dg|_{p_0}$ will drop rank, admitting more infinitesimal mechanisms.  For the 3-prism, an infinitesimal flex vector $v \in \mathbb{R}^{nd}$ with $dg|_{p_0}(v) = 0$ is pictured on the right of Figure \ref{fig:possible-flex-3prism} by arranging its $nd = 18$ components as $3$-vectors attached to each of the $n=6$ nodes.

From the discussion above, we see that infinitesimal rigidity is a sufficient condition for local rigidity but is not necessary. The complete picture comes from the local geometry of the configuration space $V_\mathbb{R}(g)$ around $p_0$.  A flex is nothing more than a path through configuration space starting at $p_0$ which is not a rigid motion. The configuration $p_0$ is locally rigid exactly when the local real dimension of $V_{\mathbb{R}}(g)$ is $\binom{d+1}{2}$, the dimension of the Euclidean group of rigid motions. See \cite{WHS2011} for more details. Thus, a technique for determining the local real dimension of an algebraic set would completely solve the problem of determining local rigidity. No such technique currently exists. 

A partial answer to the problem of determining local real dimension comes from the local dimension test described in \cite{WHS2011}. With minor caveats, the local dimension test determines the local complex dimension successfully. The real and complex local dimensions agree whenever $p_0$ is a smooth point on the algebraic set. Thus, for smooth points \cite{WHS2011} solves the local rigidity problem. However, many of the examples we are interested in are singular configurations. For singular $p_0$ like the 3-prism of Figure \ref{fig:possible-flex-3prism}, the real local dimension may differ from the complex local dimension computed in \cite{WHS2011}. The authors acknowledge this point and leave it open for future work. 

To understand the local geometry of $V_\mathbb{R}(g)$ around $p_0$ we rely on tools from \textit{numerical algebraic geometry} - specifically polynomial homotopy continuation. Given a square system of polynomial equations (square means same number of variables as equations), polynomial homotopy continuation can reliably compute \textit{all isolated solutions}, in contrast to other numerical methods like Newton iteration which can only reliably compute \textit{some solutions}. Numerical algebraic geometry draws on results from algebraic geometry to provide such guarantees. For general references see any of \cite{Wampler2013, Hauenstein2017, SommeseWamplerTEXT}. Here we describe only the basics necessary to understand the results of this article.

Given a system $g:\mathbb{C}^N \to \mathbb{C}^N$ whose solutions we would like to discover, the key is to produce a \textit{starting system} $f:\mathbb{C}^N \to \mathbb{C}^N$ whose solutions are known, and a \textit{homotopy} $H(x,t):\mathbb{C}^N \times \mathbb{C} \to \mathbb{C}^N$ such that $H(x,1)=f(x)$ and $H(x,0)=g(x)$. The key computational step is to solve a \textit{Davidenko} initial value problem for the system of ODE's resulting from differentiating $H(x(t),t) = 0$ with respect to time. Here, $x(t)$ is the smooth path of some known solution $f(x(1))=0$ to some unknown solution $g(x(0)) = 0$. Solving the ODE from the initial conditions at $H(x(1),1)=0$ to the \textit{endpoint} \cite[Chapter 10]{SommeseWamplerTEXT} at $t=0$ using any of the standard ODE solvers is called \textit{path-tracking} \cite[Section 2.3]{SommeseWamplerTEXT}. For excellent examples see the website of the \texttt{julia} implementation \texttt{HomotopyContinuation.jl}, which we use in all our examples \cite{BT2018}.

A \textit{total-degree homotopy} always provides a valid starting system \cite[Section 8.4.1]{SommeseWamplerTEXT}. In many cases there are smaller starting systems which require tracking less paths, and hence less computation \cite[Section 8.4, 8.5]{SommeseWamplerTEXT}. We demonstrate this for the 3-prism in Example \ref{example:three-prism} later. In addition to finding isolated solutions, the concept of \textit{witness sets} can be used to understand positive-dimensional components of the solution set to polynomial systems \cite[Chapter 13]{SommeseWamplerTEXT}. Say $X$ is a $k$-dimensional irreducible component of the solution set to $g:\mathbb{C}^N \to \mathbb{C}^m$. A witness set for $X$ is a tuple $(g,L,W)$ where $g$ is the polynomial system, $L$ is a linear space of dimension $N-k$, and $W$ is a list of the finitely many isolated solutions obtained by intersecting $X$ with the linear space $L$. Once a witness set for $X$ has been computed, one can easily sample more points from $X$ by perturbing $L$ to nearby linear spaces $L'$, following the solutions in $W$. The size $|W|$ is the \textit{degree} of $X$. Numerical algebraic geometry is a powerful tool that we use in this article.

\section{Changing coordinates to a moving frame}\label{section:pinned-nodes-vs-set-locations}

Our goal in this section is to remove rigid motions without changing the space of flexes. Often, when building structures, you attach or fasten them to the ground, or to the wall, or to another structure. In that case, the correct model will treat those nodes as fixed, and not allow them to move at all. This deletes node variables from the equations $g$ and columns from the Jacobian matrix $dg$. All our results still apply, though the change of coordinates we describe presently would not be necessary.

However, if you are considering a structure which you do not intend to fasten to the ground, like NASA's Super Ball Bot of Figure \ref{fig:NASA-super-ball-bot},
% \begin{center}
%     \includegraphics[width=0.4\textwidth]{figures/NASA_SUPERball_Tensegrity_Lander_Prototype.jpg}
% \end{center}
then the correct model does not have variables deleted from the system $g$. We would like handle both situations equally well.
\begin{figure}[!htb]
    \centering
    \includegraphics[width=0.3\textwidth]{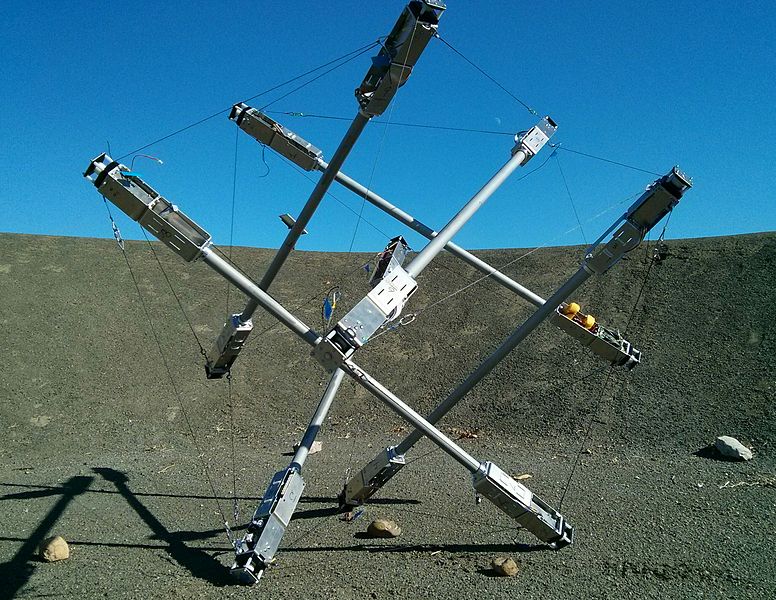}
    \caption{Structure with no nodes fastened to the ground}
    \label{fig:NASA-super-ball-bot}
\end{figure}
%The Jacobian $dg$ should be of full size, and the infinitesimal mechanisms should include the full $\binom{d+1}{2}$ space of infinitesimal rigid motions.
With no nodes fixed, the configuration space includes the full $\binom{d+1}{2}$ dimensional space of rigid motions. These rigid motions are independent of the structure being studied and are not interesting. We would like to remove them so we can focus on the flexes. What follows is a convenient procedure for removing the rigid motions. This can be viewed as a quotient construction where we work with flex representatives after modding out by rigid motions.

Observe that given a valid configuration $(x_{ik}) \in \mathbb{R}^{nd}$ which satisfies the member constraints $g$, any sequence of rigid motions applied to $(x_{ik})$ will yield  another valid configuration $(\widehat{x_{ik}})$. We use this to choose a convenient reference frame for our initial configuration and stay with this reference frame as we follow a flex.

As an example, we illustrate our choice of reference frame for a configuration embedded in $\mathbb{R}^3$. Let nodes 1,2,3 be noncollinear. Translate the reference frame until node 1 is at the origin. Next, rotate the reference frame until node 2 lies on the $x$-axis. Finally, rotate the reference frame about the $x$-axis until node 3 has $z$-coordinate zero. The resulting configuration has $\text{node 1} = (0,0,0)$, $\text{node 2} = (x_{21}, 0,0)$, and $\text{node 3} = (x_{31}, x_{32}, 0)$. We require a flex to stay in this reference frame. That is, $x_{11} = x_{12} = x_{13} = x_{22} = x_{23} = x_{33} = 0$ at all points along a path through $V_\mathbb{R}(g)$. We claim these constraints remove the rigid motions but do not change the space of flexes. To see this, note any valid configuration $(x_{ik})$ can be sent, via rigid motions, to another configuration $(\widehat{x_{ik}})$ which satisfies the above constraints. We refer to this change of coordinates as the moving frame. The general case is presented in the following two theorems.

\begin{theorem}\label{theorem:setting-locations}
Let the affine span of nodes $1,\dots,d$ be $d-1$ dimensional. Then there exists a rigid motion which changes coordinates $p_{ik}$ to $\widehat{p_{ik}}$ such that for $i \in [d]$ we have $\widehat{p_{ik}} = 0$ if $k \geq i$. We denote the new configuration $\widehat{p_0}$.
\end{theorem}

\begin{proof}
We proceed by induction on the number of nodes $j$. Let the dimension of ambient space be fixed at $d$. If $j=1$, apply a translation that moves node $1$ to the origin. Then its new coordinates satisfy $\widehat{p_{1k}} = 0$ if $k \geq 1$ (all its coordinates are zero). Now say Theorem \ref{theorem:setting-locations} holds for $j=d-1$. We find a rigid motion that fixes nodes $1, \dots ,d-1$ and rotates node $d$ such that $p_{dd} = 0$. Let $H$ be the subspace of $\mathbb{R}^d$ spanned by the first $d-2$ coordinate axes and  $K$ be the subgroup of $SO(\mathbb{R}^d)$ that fixes $H$. By induction, nodes $1, \dots, d-1$ are contained in $H$. $K$ is isomorphic to $SO(\mathbb{R}^2)$ so we can select a rotation $r \in K$ that fixes indices $1, \dots , d-2$ of every node and rotates node $d$ such that $p_{dd} = 0$ as desired.
\end{proof}

\begin{remark}\label{remark:always-d-incident-edges}
We view $\widehat{p_0}$ as an element of $\mathbb{R}^N$ where $N = nd - \binom{d+1}{2}$ by dropping the newly zero coordinates.
\end{remark}

\begin{definition}\label{definition:set-location-member-constraint-equations}
We associate to $\widehat{p_0}$ a system of equations called the \textit{moving frame member constraints} denoted
\begin{equation*}
    \widehat{g}:\mathbb{C}^N \to \mathbb{C}^{m}.
\end{equation*}
As with $g$, this system enforces the requirement that edges must have constant length and its definition is analogous to that of $g$. We simply drop the unnecessary variables.
\end{definition}

\begin{theorem}\label{theorem:setting-locations-preserves-flexes}
There exists a flex $p:[0,1] \to \mathbb{R}^{nd}$ of the initial configuration $p_0$ if and only if there exists a flex $\widehat{p(t)}:[0,1] \to \mathbb{R}^N$ of the initial configuration $\widehat{p_0}$.
\end{theorem}

\begin{proof}
Order the nodes such that nodes $1, \dots d$ span $d-1$ dimensions in $\mathbb{R}^{nd}$. Let $\pi:\mathbb{R}^{nd} \to \mathbb{R}^N$ be the natural projection and $\iota: \mathbb{R}^N \to \mathbb{R}^{nd}$ be the natural injection. Let $r$ be the rigid motion of Theorem \ref{theorem:setting-locations} such that $\pi (r \cdot p_0) = \widehat{p_0}$. First, say we have a flex $\widehat{p(t)}:[0,1] \to \mathbb{R}^N$. Then $r^{-1} \cdot \iota(\widehat{p(t)})$ is a flex of $p_0$ in $\mathbb{R}^{nd}$. Now suppose we have a flex $p:[0,1] \to \mathbb{R}^{nd}$. By Theorem \ref{theorem:setting-locations} for every point $p(t)$, there exists a rigid motion $r(t)$ which sends $p(t)$ to another point with $\widehat{p(t)}_{ik} = 0$ if $k \geq i$ and $i \in [d]$. The continuity of $p(t)$ implies $r(t)$ is continuous. Thus $\pi(r(t) \cdot p(t))$ is a flex of $\widehat{p_0}$.
\end{proof}

For interested readers, we make note of the connection to the method of \textit{moving frames} \cite{OlverMovingFrames}. Understood in this sense, we are making our calculations in a \textit{coordinate cross-section}.

\section{Epsilon local rigidity}\label{section:epsilon-local-rigidity}

In this section we determine whether a given initial configuration $p_0$ is $\varepsilon$-locally rigid. We fix $N = nd - \binom{d+1}{2}$. It is tempting to define a notion of $\varepsilon$-locally rigid in the original configuration space $\mathbb{R}^{nd}$ but since any configuration can be translated by an arbitrary amount, the better notion is applied to the point after changing to the moving frame $\widehat{p_0} \in \mathbb{R}^N$ as in Theorem \ref{theorem:setting-locations}. This motivates the following

\begin{definition}\label{definition:epsilon-locally-rigid}
Let $p_0$ be an initial configuration and $\widehat{p_0}$ be the configuration in the moving frame of Theorem \ref{theorem:setting-locations}. We say that $p_0$ is \textit{$\varepsilon$-locally rigid} if every flex $\widehat{p(t)}$ of $\widehat{p_0}$ satisfies $\widehat{p(t)} \in B_{\varepsilon}(\widehat{p_0})$ for all $t \in [0,1]$, where $B_{\varepsilon}(\widehat{p_0})$ is the open $\varepsilon$-ball centered at $\widehat{p_0}$.
\end{definition}

\begin{remark}\label{remark:epsilon-locally-rigid-meaning}
In other words, if $p_0$ is $\varepsilon$-locally rigid, then any positive-dimensional connected component of the real algebraic set $V_{\mathbb{R}}(\widehat{g})$ containing $\widehat{p_0}$ stays within some ball about $\widehat{p_0}$. Any flex that may exist can be safely ignored if $\varepsilon$ is sufficiently small. For all practical purposes, it is as if $p_0$ is locally rigid.
\end{remark}

We now move towards Theorem \ref{theorem:epsilon-local-rigidity} and Algorithm \ref{algorithm:epsilon-local-rigidity} to decide whether a configuration $p_0$ is $\varepsilon$-locally rigid.

\begin{definition}\label{definition:epsilon-member-constraint-system}
Let $\widehat{g} = [\widehat{g_1},\dots,\widehat{g_m}]^T$ be the moving frame member constraints associated to $\widehat{p_0}$ as in Definition \ref{definition:set-location-member-constraint-equations}. Define the polynomial system $\widehat{g}_\varepsilon: \mathbb{C}^N \to \mathbb{C}$
\begin{equation*}
    \widehat{g}_{\varepsilon} = \widehat{g_1}^2 + \cdots + \widehat{g_m}^2 + s_{\varepsilon}^2,
\end{equation*}
where $s_{\varepsilon}$ is defined by
\begin{equation*}
    s_{\varepsilon} = \varepsilon^2 - \sum_{k=1}^d \sum_{i=1}^n (x_{ik} - p_{ik})^2.
\end{equation*}
The system of $\varepsilon$-\textit{member constraints} associated to $\widehat{p_0}$ is given by $\widehat{g}_\varepsilon(x) = 0$. We denote the corresponding algebraic set $\widehat{V}_\varepsilon := \{x \in \mathbb{C}^N : \widehat{g}_\varepsilon(x) = 0\}$.
\end{definition}

\begin{lemma}\label{lemma:codimension-one-irreducible-components}
The irreducible components of $\widehat{V}_\varepsilon$ are of dimension exactly $N-1$.
\end{lemma}

\begin{proof}
By Theorem 13.4.2 of \cite{SommeseWamplerTEXT}, the possible dimensions of irreducible components $X$ of an algebraic set $V(f)$ for $f:\mathbb{C}^N \to \mathbb{C}^n$ are bounded between
\begin{equation*}
    N - \text{rank }f \leq \text{dim } X \leq N-1,
\end{equation*}
where the rank of $f$ is the dimension of the closure of its image as a map, or equivalently, the generic rank of its Jacobian. For a single, nonzero polynomial like $\widehat{g}_\varepsilon:\mathbb{C}^{N} \to \mathbb{C}$ we have that $N - 1 \leq \text{dim } X \leq N -1$.
\end{proof}

Below we will prove Theorem \ref{theorem:epsilon-local-rigidity}, which follows from Theorem 5 of \cite{Hauenstein2012}. But first we will state formally the assumptions required for the theorem. We also note that Theorem 5 of \cite{Hauenstein2012} draws on results from \cite{AubryRouillierSafeyElDin2002, RouillierRoySafeyElDin2000} and also from the 1954 paper of Seidenberg \cite{Seidenberg1954}.

\begin{assumption}\label{assumption:for-theorem-5-hauenstein}
We collect here the following list of assumptions which refer to the homotopy $H(x,\lambda,t)$ defined in Theorem \ref{theorem:hauenstein-2012-thm5} below.
\begin{enumerate}
    \item Let $N > k >0$ and $f:\mathbb{R}^N \to \mathbb{R}^{N-k}$ be a polynomial system with real coefficients, with $V \subset V(f)$ a pure $k$-dimensional algebraic set with witness set $\{f,L,W\}$.
    \item Assume that the starting solutions to $H(x,\lambda,1)=0$ are finite and nonsingular.
    \item Assume also that the number of starting solutions is equal to the maximum number of isolated solutions to $H(x,\lambda,1)=0$ as $z,\gamma,y,\alpha$ vary over $\mathbb{C}^{N-k} \times \mathbb{C} \times \mathbb{C}^N \times \mathbb{C}^{N-k+1}$. This will be true for a nonempty Zariski open set of $\mathbb{C}^{N-k} \times \mathbb{C} \times \mathbb{C}^N \times \mathbb{C}^{N-k+1}$.
    \item Assume all the solution paths defined by $H$ starting at $t=1$ are trackable. This means that for each starting solution $(x^*,\lambda^*)$ there exists a smooth map $\xi:(0,1] \to \mathbb{C}^N \times \mathbb{C}^{N-k+1}$ with $\xi(1) = (x^*,\lambda^*)$ and for all $t \in (0,1]$ we have $\xi(t)$ is a nonsingular solution of $H(x,\lambda,t)$.
    \item Assume that each solution path converges, collecting the endpoints of all solution paths in the sets $E$ and $E_1 = \pi(E)$ where $\pi(x,\lambda) = x$ projects onto the $x$ coordinates, forgetting the $\lambda$ coordinates. 
\end{enumerate}
\end{assumption}

\begin{theorem}[Theorem 5 of \cite{Hauenstein2012}]\label{theorem:hauenstein-2012-thm5}
Suppose that the conditions in Assumption \ref{assumption:for-theorem-5-hauenstein} hold. Let $z \in \mathbb{R}^{N-k}$, $\gamma \in \mathbb{C}$, $y \in \mathbb{R}^N - V_{\mathbb{R}}(f)$, $\alpha \in \mathbb{C}^{N - k + 1}$, and $H:\mathbb{C}^N \times \mathbb{C}^{N-k+1} \times \mathbb{C} \to \mathbb{C}^{2N-k+1}$ be the homotopy defined by
\begin{equation}
    H(x,\lambda,t) = \left[ \begin{array}{c}
        f(x) - t \gamma z  \\
        \lambda_0(x-y) + \lambda_1 \nabla f_1(x)^T + \cdots + \lambda_{N-k} \nabla f_{N-k}(x)^T \\
        \alpha^T \lambda - 1
    \end{array} \right]
\end{equation}
where $f(x) = [f_1(x), \dots, f_{N-k}(x) ]^T$. Then
\begin{equation*}
    E_1 \cap V \cap \mathbb{R}^N
\end{equation*}
contains a point on each connected component of $V_\mathbb{R}(f)$ contained in $V$.
\end{theorem}

\begin{theorem}\label{theorem:epsilon-local-rigidity}
Let $p_0$ be an initial configuration and $\widehat{g_\varepsilon}$ be according to Definition \ref{definition:epsilon-member-constraint-system} above. Taking $f = \widehat{g_\varepsilon}$ in Theorem \ref{theorem:hauenstein-2012-thm5}, we find that if conditions two through five in Assumption \ref{assumption:for-theorem-5-hauenstein} are met, then
\begin{equation*}
    E_1 \cap \widehat{V}_\varepsilon \cap \mathbb{R}^N = \emptyset
\end{equation*}
implies that $p_0$ is $\varepsilon$-locally rigid.
\end{theorem}

\begin{proof}
Take $f = \widehat{g_\varepsilon}$ and $V = \widehat{V}_\varepsilon$ in the notation of Theorem \ref{theorem:hauenstein-2012-thm5} above. We have $N - k = 1$ and by Lemma \ref{lemma:codimension-one-irreducible-components} all irreducible components are of dimension $k = N - 1$. Therefore, the first condition of Assumption \ref{assumption:for-theorem-5-hauenstein} is met.
Say $E_1 \cap \widehat{V}_\varepsilon \cap \mathbb{R}^N = \emptyset$ but $p_0$ is \textit{not} $\varepsilon$-locally rigid. Let $\widehat{p}(t)$ be a flex such that $\widehat{p}(1) \notin B_{\varepsilon}(\widehat{p_0})$ and let $P = \widehat{p}([0,1])$ be the image of $\widehat{p}$. Then
\begin{equation*}
    \Big(P \cap B_\varepsilon(\widehat{p_0}) \Big) \cup \Big(P \cap \overline{B_\varepsilon(\widehat{p_0})}^{\,c} \Big)
\end{equation*}
%$(P \cap B_\varepsilon(\widehat{p_0})) \cup (P \cap \overline{B_\varepsilon(\widehat{p_0})}^{\,c})$
is a separation of $P$ contradicting the continuity of $\widehat{p}$. 
\end{proof}

Theorem \ref{theorem:epsilon-local-rigidity} suggests the following Algorithm \ref{algorithm:epsilon-local-rigidity}.

\vspace{3pt}
\begin{algorithm}[H]
    \DontPrintSemicolon
    \KwIn{Initial configuration $p_0 \in \mathbb{R}^{nd}$, edge set $E$, and choice of $\varepsilon > 0$.}
    \KwResult{Boolean $v$ which is \textit{true} if items 2, 4, and 5 of Assumption \ref{assumption:for-theorem-5-hauenstein} are satisfied. Boolean $u$ which is \textit{true} if the set $E_1 \cap \widehat{V} \cap \mathbb{R}^N$ of Theorem \ref{theorem:epsilon-local-rigidity} is empty, and false otherwise. Set $R$ which may be empty or else contains at least one point on each connected component of $V_{\mathbb{R}}(\widehat{g_\varepsilon})$.}
    
    Apply the rigid motions of Theorem \ref{theorem:setting-locations} to $p_0$ obtaining $\widehat{p_0} \in \mathbb{R}^N$ for $N = nd - \binom{d+1}{2}$.\;
    Form the systems of equations $\widehat{g_\varepsilon}$ according to Definition \ref{definition:epsilon-member-constraint-system}.\;
    Calculate a witness set $W$ for the pure $N-1$ dimensional algebraic set $V(\widehat{g_\varepsilon}) \subset \mathbb{C}^N$. \;
    Produce $z,\gamma,y,\alpha$ such that item 3 of Assumption \ref{assumption:for-theorem-5-hauenstein} holds. \;
    Use the algorithm presented in Section 2.1 of \cite{Hauenstein2012}, obtaining the boolean $v$ and the set of real solutions $R$.\;
    If $R$ is the empty set, set $u$ as true, else set $u$ as false.\;
    Output the booleans $v$ and $u$, and the set $R$.\;
    \caption{Epsilon local rigidity}
    \label{algorithm:epsilon-local-rigidity}
\end{algorithm}
\vspace{12pt}

\begin{remark}\label{remark:genericity-assumptions} An appropriate choice of $y \in \mathbb{R}^N \setminus V_{\mathbb{R}}(\widehat{g_\varepsilon})$ could be $p_0$ itself, or $p_0 + \mathcal{N}(0,\sigma^2)$ for some random multivariate Gaussian noise with mean zero and variance $\sigma^2$. In step 4 above, if items 2, 3, 4, or 5 of Assumption \ref{assumption:for-theorem-5-hauenstein} fail to hold, then generating new and random points $z,\gamma,y,\alpha$ could be required. 
\end{remark}

\begin{example}\label{example:slingshot-epsilon-locally-rigid}
As an illustrative example, we would like a configuration which we know to be locally rigid, but fails to be infinitesimally rigid so that Theorem \ref{thm:FundamentalTheorem1} does not apply, and which is also a singular point of the member constraints so that the local real dimension may differ from the local complex dimension computed in \cite{WHS2011}. As a simple example, consider a graph on nodes $[5]=\{1,2,3,4,5\}$ with edges $E = \{12,13,14,23,24,34,45 \}$. Consider the configuration in Figure \ref{fig:slingshot}.
\begin{figure}[!htb]
    \centering
    \includegraphics[width=0.25\textwidth]{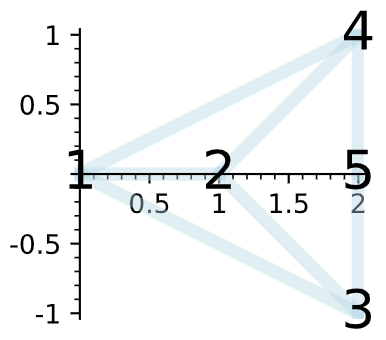} 
    \caption{}
    \label{fig:slingshot}
\end{figure}
As can be verified by calculating the nullspace of the Jacobian at generic points and at this point, this configuration is singular. Its Jacobian admits an additional infinitesimal flex, and hence Theorem \ref{thm:FundamentalTheorem1} does not apply. However, it is also simple enough that we can see it is locally rigid: Since the triangles among nodes $123$ and $124$ are both rigid, the only node that could possibly move in a continuous flex is node $5$. However, we can also see that node $5$ is restricted by node $3$ to move in one circle of radius equal to the edge length of edge $35$, while by node $4$ it is restricted to another similar circle. These circles intersect in exactly one point, the location of node $5$ in our initial configuration $p_0$. Thus $p_0$ is locally rigid. 

We implemented Algorithm \ref{algorithm:epsilon-local-rigidity} for this example using \texttt{HomotopyContinuation.jl} in \texttt{julia}, obtaining $\varepsilon$-local rigidity for $\varepsilon \in \{ 0.1, 0.01, 0.001, 0.0001 \}$. The point $y$ was generated by Gaussian noise applied to the original configuration $p_0$, then paths were tracked, obtaining zero real solutions. For each value of $\varepsilon$ all paths were trackable and the items in Assumption \ref{assumption:for-theorem-5-hauenstein} were satisfied, but none of the resulting solutions were real-valued. Therefore, we can conclude by Theorem \ref{theorem:epsilon-local-rigidity} that this configuration $p_0$ is $\varepsilon$-locally rigid for $\varepsilon \in \{ 0.1, 0.01, 0.001, 0.0001 \}$. Other choices of $\varepsilon$ can be made by another user according to their application.

\begin{figure}[!htb]
    \centering
    \includegraphics[width=0.45\textwidth]{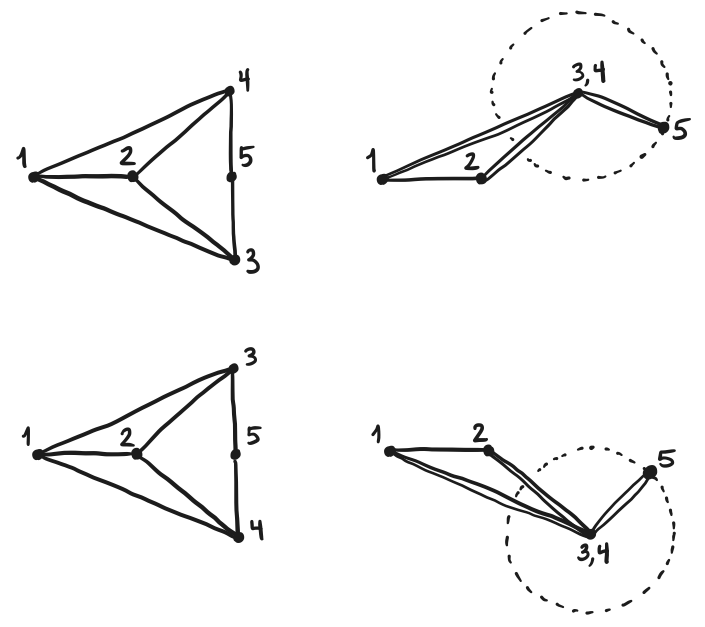}
    \caption{Several singular configurations}
    \label{fig:several-singular-slingshots}
\end{figure}

The example we just described is a singular configuration depicted in Figure \ref{fig:several-singular-slingshots} (left). Now consider the singular configuration depicted in Figure \ref{fig:several-singular-slingshots} (right). This configuration is clearly not locally rigid. However, again it is singular and admits infinitesimal flexes, hence Theorem \ref{thm:FundamentalTheorem1} does not apply. In this singular configuration, nodes $3$ and $4$ coincide, yielding a flexible node $5$ which can freely move around a circle centered at nodes $3,4$. This configuration is obviously flexible, and when we run Algorithm \ref{algorithm:epsilon-local-rigidity} we obtain exactly two real solutions (after projecting away the $\lambda$ components). Upon examination, these two real solutions correspond to node $5$ moving in either possible direction along a circle centered at nodes $3,4$, as expected. Algorithm \ref{algorithm:discrete-flex} described in Section \ref{section:producing-a-discrete-flex} may be applied to visualize this flex.
\end{example}

\begin{example}\label{example:three-prism}
Consider the 3-prism of Figure \ref{fig:possible-flex-3prism}. We have nodes $[6]=\{1,2,3,4,5,6\}$ and edges $E = \left\{ 12, 13, 14, 15, 23, 25, 26, 34, 36, 45, 46, 56 \right\}$. As mentioned previously, for generic configurations the 3-prism is infinitesimally rigid and therefore rigid. However, for the configuration $p_0 \mapsto \widehat{p_0}$
\begin{equation}\label{equation:embedding-p0-3prism}
    p_0 = \left[ \begin{array}{ccc}
        p_{11} & p_{12} & p_{13}\\ p_{21} & p_{22} & p_{23}\\ p_{31} & p_{32} & p_{33}\\ p_{41} & p_{42} & p_{43}\\ p_{51} & p_{52} & p_{53}\\ p_{61} & p_{62} & p_{63}
    \end{array} \right] = \left[ \begin{array}{ccc}
        1 & 0 & 0\\ -\frac{1}{2} & \frac{\sqrt{3}}{2} & 0\\  -\frac{1}{2} & -\frac{\sqrt{3}}{2} & 0\\ -\frac{\sqrt{3}}{2} & -\frac{1}{2} & 3\\ \frac{\sqrt{3}}{2} & -\frac{1}{2} & 3\\ 0 & 1 & 3
    \end{array} \right] \mapsto 
\left[
\begin{array}{ccc}
0.0 & 0.0 & 0.0 \\
1.7320508 & 0.0 & 0.0 \\
0.8660254 & -1.5 & 0.0 \\
1.3660254 & -1.3660254 & 3.0 \\
-0.133975 & -0.5 & 3.0 \\
1.3660254 & 0.3660254 & 3.0 \\
\end{array}
\right]
\end{equation}
there is an infinitesimal flex (right side of Figure \ref{fig:possible-flex-3prism}) which is linearly independent of the $\binom{d+1}{2}=6$ infinitesimal rigid motions. Therefore Theorem \ref{thm:FundamentalTheorem1} fails to apply.

Applying Algorithm \ref{algorithm:infinitesimal-flex-parameter-homotopy} to gather some preliminary numerical evidence for how this configuration might deform, we obtain many approximate, numerical configurations near $p_0$, including, for example
\begin{equation}\label{equation:downwards}
\left[
\begin{array}{ccc}
0.0 & 0.0 & 0.0 \\
1.734502 & 0.0 & 0.0 \\
0.868440 & -1.499228 & 0.0 \\
1.434394 & -1.322820 & 2.986758 \\
-0.127807 & -0.572257 & 2.988424 \\
1.303288 & 0.404336 & 2.986506 \\
\end{array}
\right].
\end{equation}
This suggests the configuration $p_0$ can twist downwards, since the $z$ coordinate of nodes $4,5,6$ falls from $3.0$ to approximately $2.986506$. As discussed in the next Section \ref{section:producing-a-discrete-flex}, the configurations produced by Algorithm \ref{algorithm:infinitesimal-flex-parameter-homotopy} may not correspond to exact, real-valued solutions of the system of member constraints. Without further exploration we treat them as experimental and numerical evidence for what may happen. See Remark \ref{remark:numerical-configurations-meaning} for a discussion of the possible relevance and meaning of such numerical near-solutions. However, this motivates our application of Algorithm \ref{algorithm:epsilon-local-rigidity} as an attempt to rule out these configurations as being part of a continuous flex from $p_0$ which exactly satisfies the member constraints.

Therefore we apply Algorithm \ref{algorithm:epsilon-local-rigidity} to $p_0$. A naive total degree homotopy \cite[Section 8.4.1]{SommeseWamplerTEXT} would require tracking $67,108,864$ paths. However, using mixed volumes of Newton polytopes \cite{HuberSturmfels1995PolyhedralMethodForSolvingSparsePolynomialSystems} lowers the path-tracking count to just $1,062,880$. This calculation took only a few hours on a personal computer and the \texttt{julia} code required is available at \cite{heaton-tensegrity-trusses-github}, which also includes code for an expository article \cite{H2019} which treats the 3-prism example in detail. In this calculation, we obtain $\varepsilon$-local rigidity for $\varepsilon=0.1$, even though the computed configurations like (\ref{equation:downwards}) are outside such a small $\varepsilon$-ball. Therefore, there can be no real-valued flex connecting $p_0$ and (\ref{equation:downwards}). In fact, since the 3-prism is famous, the rigidity of this initial configuration $p_0$ is well-known, thus our calculations agree with previous results.
\end{example}

\section{Producing a discrete flex}\label{section:producing-a-discrete-flex}
In this section we describe an algorithm that repeatedly solves a system of parametrized polynomial equations in order to produce a sequence of real-valued, valid configurations. If a continuous flex of $p_0$ exists, then this procedure will produce a discrete sampling of points from that continuous flex. The resulting sequence of configurations may be plotted and animated, yielding easily understandable information for the scientist. In future work, we plan to implement this algorithm in a freely available \texttt{julia} package, utilizing the existing algorithms of the \texttt{julia} package \texttt{HomotopyContinuation.jl} \cite{BT2018}, which implements polynomial homotopy continuation as discussed in Section \ref{section:preliminaries}. A main goal for our package will be ease of use. This is currently under development, but a rough example of its output is shown in Figure \ref{fig:cube-flexing}.
\begin{figure}[!htb]
    \centering
    \includegraphics[width=0.45\textwidth]{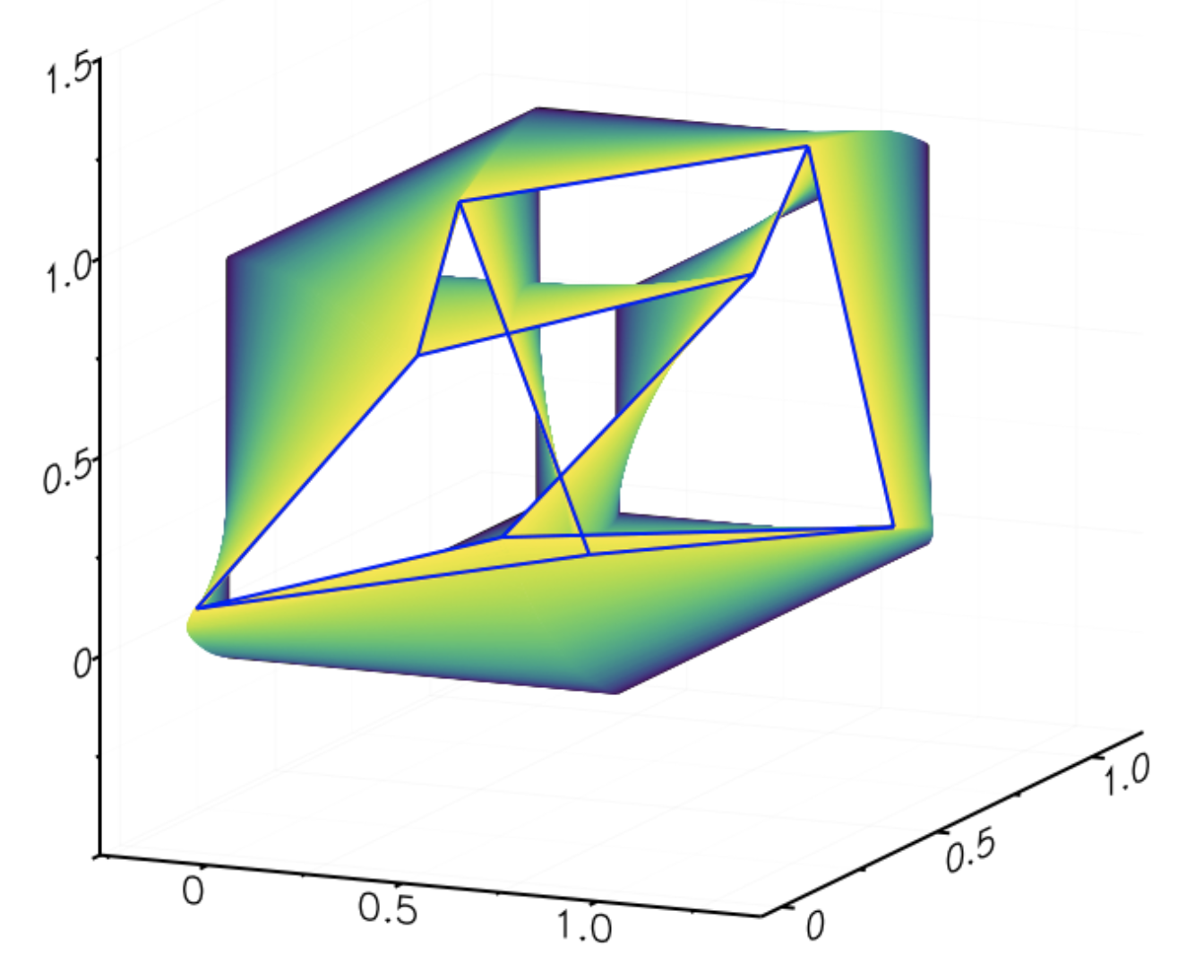}
    \caption{100 configurations of a cube flexing}
    \label{fig:cube-flexing}
\end{figure}
These are images of a discrete flex with $M=100$ points computed using homotopy continuation on a cube deforming freely. The cube is obviously flexible, and the goal is to implement software which will find more surprising flexes from other, more complicated examples. We also note that other excellent software exists for homotopy continuation, including \cite{Bertini, Hom4PS, PHCpack}.

\vspace{3pt}
\begin{algorithm}[H]
    \DontPrintSemicolon
    \KwIn{Initial configuration $p_0 \in \mathbb{R}^{nd}$, edge set $E$, choice of $\varepsilon_0 > 0$ and $M \in \mathbb{Z}_{>0}$.}
    \KwResult{A discrete flex in the form of a list $P$ of configurations $p_1,p_2,\dots,p_{M}$ to be animated and visualized, or potentially the message of $\varepsilon$-local rigidity for some $\varepsilon = j \cdot \varepsilon_0$.}
    
    Initialize a list with one element $P = [p_0]$ to be filled with more points $p_j \in \mathbb{R}^N$ as in $p_1,p_2,\dots,p_M$ if the algorithm succeeds.\;
    \For{ j in 1:M}{
        Set $\varepsilon := j \cdot \varepsilon_0$.\;
        Apply Algorithm \ref{algorithm:epsilon-local-rigidity} with inputs $p_0, E, \varepsilon$, collecting the outputs $u,v,R$.\;
        \eIf{u = 1}{
            Output the current list $P$ and the message that $p_0$ is $\varepsilon$-locally rigid.\;}
            {
            Collect the output set $R$ and store it as $R_j = R$.\;
            Set $p_{j} := \text{argmin }\, \{ \text{dist}(p_{j-1},w) \, : \, w \in R_{j} \}$. \;
            Append $p_j$ to the list $P$. \;
        }
    }
    Return the list $P$ as well as an animation of each of its $M$ configurations displayed in $\mathbb{R}^d$ if $d=2,3$.\;
    \caption{A discrete flex}
    \label{algorithm:discrete-flex}
\end{algorithm}
\vspace{12pt}

There are many possible alterations of the above algorithm, which we plan to explore in our implementation. First, the 2-homogeneous structure should be exploited in generating start systems. Second, a parameter homotopy could be used after obtaining the first new solution $p_1$. Consider the square system of equations
\begin{equation*}
    F_{y,\varepsilon} (x_1, \dots, x_N, \lambda_0, \lambda_1) = \left[ \begin{array}{c}
        \widehat{g_\varepsilon}(x)  \\
        \lambda_0 (x-y) + \lambda_1 \nabla \widehat{g_\varepsilon}\\
        \alpha_0 \lambda_0 + \alpha_1 \lambda_1 - 1 
    \end{array} \right] : \mathbb{C}^{N+2} \to \mathbb{C}^{N+2}.
\end{equation*}
After solving and finding a new configuration $p_1$ at radius $\varepsilon$ away from $p_0$, this means we have obtained solutions $(x,\lambda)$ to $F_{y,\varepsilon}$ for some specific $y \in \mathbb{R}^N$ and $\varepsilon > 0$, where one of these solutions has $x = p_1$. We could then consider the homotopies perturbing $y$ to $y^\prime$ or perturbing $\varepsilon$ to $\varepsilon^\prime$ as in
\begin{equation}
    H_y(x,\lambda,t) = F_{(1-t)y^\prime + ty,\varepsilon}(x,\lambda)
\end{equation}
or
\begin{equation}
    H_\varepsilon(x,\lambda,t) = F_{y,(1-t)\varepsilon^\prime + t\varepsilon}(x,\lambda),
\end{equation}
either of which would produce new and relevant configurations. In particular, using $H_\varepsilon$ can allow us to generate new solutions for expanding (or contracting) $\varepsilon$-balls about $p_0$, slightly modifying line 4 of Algorithm \ref{algorithm:discrete-flex}.

\begin{remark}
In the homotopies above we have removed the usual factor $\gamma$. This allows real-valued solutions like $p_1$ to stay real-valued along the parameter homotopy. This will succeed unless we cross the discriminant, so we can expect success for $|\varepsilon - \varepsilon^\prime|$ small. For small perturbations $\varepsilon$ to $\varepsilon^\prime$ we can numerically follow a continuous flex and every step of the path-tracking process will compute new and real-valued configurations sampled from that continuous flex. More precisely, starting from a non-singular solution to $F_{y,\varepsilon}$ we can continue to a non-singular solution of $F_{y,\varepsilon^\prime}$ for $|\varepsilon - \varepsilon^\prime|$ small. Even if we do cross the discriminant, it may be due to phenomena involving complex solutions elsewhere, and thus not affect our particular real-valued solution.
\end{remark}

\begin{remark}\label{remark:follow-epsilon-to-zero}
In particular, letting $\varepsilon \to 0$ we can attempt to follow any real-valued points $p_1$ on the $\varepsilon$-sphere towards $p_0$. If there is a continuous flex of $p_0$ we can expect one of the real-valued solutions on the $\varepsilon$-sphere to move towards $p_0$ as $\varepsilon \to 0$. The collection of points computed along the way are a \textit{discrete flex} of $p_0$, having been sampled from a continuous flex. We can also follow the discrete flex away from $p_0$ by letting $\varepsilon$ increase, computing points in a parameter homotopy as above.
\end{remark}

\begin{remark}\label{remark:same-connected-component}
It is tempting to use monodromy or the trace test on each of the $p_j$ to ensure they are points on the same irreducible component. However, even if they are on the same irreducible component, they could be on distinct connected components of the real algebraic set. Therefore, applying the trace test or monodromy are necessary but not sufficient conditions for our discrete flex to be sampled from the same connected component of $V_{\mathbb{R}}(\widehat{g})$. This option could be added to Algorithm \ref{algorithm:discrete-flex} explicitly, but this would require witness sets for each irreducible component of $V(\widehat{g})$ be computed, whereas here we are only assuming computation of a witness set for the pure $(N-1)$-dimensional component of $V(\widehat{g_\varepsilon})$.
\end{remark}

Here we will describe another useful application of homotopy continuation to the problem of rigidity. Presented with a new example, this would be the first algorithm to try. The output of Algorithm \ref{algorithm:infinitesimal-flex-parameter-homotopy} gives information about how the example may deform, should flexes exist, and can be applied even to much larger examples. The only requirement is that you calculate an infinitesimal flex, since Algorithm \ref{algorithm:infinitesimal-flex-parameter-homotopy} specifically searches in that direction of configuration space.

\vspace{3pt}
\begin{algorithm}[H]
    \DontPrintSemicolon
    \KwIn{Initial configuration $p_0 \in \mathbb{R}^{nd}$, edge set $E$, infinitesimal flex $v \in \mathbb{R}^{nd}$, choice of $\varepsilon_0 > 0$ and $M \in \mathbb{Z}_{>0}$.}
    \KwResult{A discrete flex in the form of a list $P$ of configurations $p_1,p_2,\dots,p_{M}$ to be animated and visualized, or potentially the message of \textit{no real solutions} for some $\varepsilon = j \cdot \varepsilon_0$.}
    
    Change coordinates to the moving frame as in Theorem \ref{theorem:setting-locations}. Let $\widehat{p_0}$ and $\widehat{v}$ be the resulting configuration and infinitesimal flex. Initialize a list with one element $P = [\widehat{p_0}]$ to be filled with more points $p_j \in \mathbb{R}^N$ as in $[\widehat{p_0},p_1,p_2,\dots,p_M]$ if the algorithm succeeds. \;
    \For{ j in 1:M}{
        Set $\varepsilon := j \cdot \varepsilon_0$, set $\ell_v = \widehat{v}^T x - \widehat{v}^T p $, and set $\ell_{v,\varepsilon} = \widehat{v}^T x - \widehat{v}^T p - \varepsilon$.\;
        Set $g$ as the polynomial system of moving frame member constraints associated to $p_0$ as in Definition \ref{definition:set-location-member-constraint-equations}.\;
        Solve the real parameter homotopy, collecting the resulting solution as $p_j$ by following the solution $p_{j-1}$ from $t=1$ to $t=0$ in
        \begin{equation*}
            h_t(x) = (1-t)\left[ \begin{array}{c}
                g \\
                \ell_{v,\epsilon}  
            \end{array} \right] + t \left[ \begin{array}{c}
                g \\
                \ell_v  
            \end{array} \right].
        \end{equation*}
        To be clear, $p_{j-1}$ solves $h_1(x)$. Follow that one solution via homotopy toward a new solution stored as $p_j$, which satisfies $h_0(x)$.\;
        \eIf{$p_j$ is complex-valued or the parameter homotopy failed}{
            Output the current list $P$ and the message that there were \textit{no real solutions} past $\varepsilon = j \cdot \varepsilon_0$.\;}
            {
            Collect the real-valued output $p_j$ and append it to the list $P$.\;
        }
    }
    Return the list $P$ as well as an animation of each of its $M$ configurations displayed in $\mathbb{R}^d$ if $d=2,3$.\;
    \caption{Infinitesimal flex parameter homotopy}
    \label{algorithm:infinitesimal-flex-parameter-homotopy}
\end{algorithm}
\vspace{12pt}

\begin{remark}\label{remark:numerical-configurations-meaning}
Here we discuss how to interpret the output of Algorithm \ref{algorithm:infinitesimal-flex-parameter-homotopy}. Although the configurations obtained in Algorithm \ref{algorithm:infinitesimal-flex-parameter-homotopy} are valid approximate solutions to the polynomial system up to standard numerical tolerances of your software, deciding whether a given point is real-valued or complex-valued is still typically done using a cutoff threshold. Therefore, the numerically computed configurations $p_1,p_2,\dots,p_M$ may truly correspond to complex-valued solutions with extremely small imaginary parts.

However, they are not without meaning. Consider Examples \ref{example:slingshot-epsilon-locally-rigid}, \ref{example:three-prism}, and \ref{example:sticky-spheres}. For Example \ref{example:slingshot-epsilon-locally-rigid}, Algorithm \ref{algorithm:infinitesimal-flex-parameter-homotopy} did not return any solutions even when we used the infinitesimal flex directions to search. This means the results correctly suggest the configuration is very rigid. For both the 3-prism and Cluster 891529 in Examples \ref{example:three-prism} and \ref{example:sticky-spheres}, whenever we used a randomly chosen vector $w$ that was not pointing in the direction of any infinitesimal flex $v$ of $p_0$, then Algorithm \ref{algorithm:infinitesimal-flex-parameter-homotopy} correctly returned no output. Only when we specifically searched toward an infinitesimal flex direction did the algorithm return real-valued solutions. Therefore, Algorithm \ref{algorithm:infinitesimal-flex-parameter-homotopy} gives valid and useful information about \textit{near-deformations}, even if those numerically computed deformations turn out to be complex-valued configurations satisfying the member constraints but with extremely small imaginary parts. This is still useful and relevant information.

If further information is required of these numerically computed configurations $p_1,p_2,\dots,p_M$, it is possible to use Smale's $\alpha$-theory to determine if the $p_i$ are so-called \textit{approximate zeros} of the given polynomial system \cite{BlumCuckerShubSmaleComplexityAndRealComputation1998, Smale1986NewtonsMethodEstimatesFromDataAtOnePoint}. Roughly, this decides whether the computed configuration is within the region of quadratic convergence of Newton's method to some \textit{exact solution} of the system of polynomials. In fact, it is also possible to decide if a given numerical solution corresponds to an exact real-valued solution. This is implemented in the software \texttt{alphaCertified}~ \cite{HauensteinSottile2012AlphaCertified}.
\end{remark}

\begin{example}\label{example:sticky-spheres}
We examine Cluster 891529 which is a so-called sticky sphere cluster. Briefly, nano and microscale particles sometimes interact only on distance ranges much smaller than their diameters, so it makes sense to model them taking this into account. For more details see \cite{Holmes-Cerfon-2017-StickySpheres}. In this realm, Cluster 891529 was shared with us by Miranda Holmes-Cerfon via personal communication. It is believed that Cluster 891529 is rigid but not second-order rigid (for the definition of second-order rigid see \cite{secondorderrigidity1980}). In particular, it is not first-order rigid (Theorem \ref{thm:FundamentalTheorem1}) since it admits infinitesimal flexes.
\begin{figure}[!htb]
    \centering
    \includegraphics[width=0.7\textwidth]{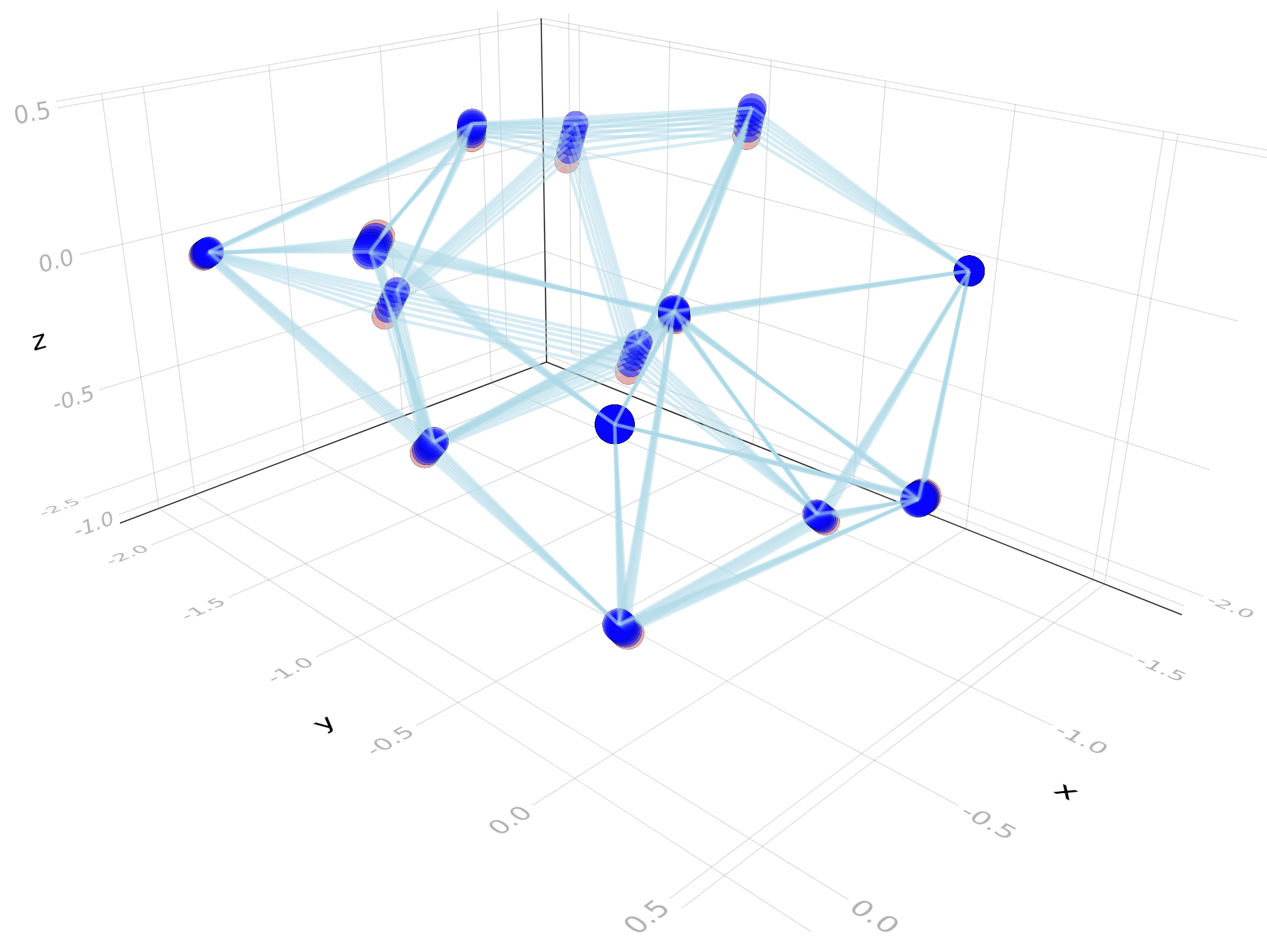}
    \caption{Algorithm \ref{algorithm:infinitesimal-flex-parameter-homotopy} applied to sticky sphere Cluster 891529}
    \label{figure:sticky-spheres}
\end{figure}
We calculated an infinitesimal flex $v$ numerically using the Jacobian of the system of member constraints, and then applied Algorithm \ref{algorithm:infinitesimal-flex-parameter-homotopy} to this initial configuration $p_0$ of Cluster 891529. The output is displayed in Figure \ref{figure:sticky-spheres}. We find experimentally that many nearby numerical configurations can be obtained using $v$, or in fact any vector that points in roughly the same direction in configuration space. This suggests that if Cluster 891529 is locally rigid, it is a singular point of of the system of member constraints of high multiplicity.
\end{example}

Unfortunately, larger examples like Cluster 891529 of Example \ref{example:sticky-spheres} are currently out of reach of Algorithm \ref{algorithm:epsilon-local-rigidity} to test for $\varepsilon$-local rigidity.
However, there is good news. \texttt{HomotopyContinuation.jl} version \texttt{2.0} will soon be finished, which will include significant changes in how polynomials are handled. In particular, instead of expanding every polynomial in the monomial basis, the inherent structure present in polynomial systems arising in applications will be preserved throughout, improving memory requirements, and the numerics of function evaluations. For example, the polynomial systems we deal with here involve many sums of squares of differences of variables. Expanding these in the monomial basis results in much higher memory requirements. When we square the entire polynomial itself, as in Definition \ref{definition:epsilon-member-constraint-system}, it is in fact much worse. There is good reason to believe this example will be within reach once this new software is finished. Nonetheless, Cluster 891529 clearly demonstrates the utility of Algorithm \ref{algorithm:infinitesimal-flex-parameter-homotopy}.

\section{Conclusion}\label{section:conclusion}

We considered bar frameworks from the perspective of numerical algebraic geometry, in particular using results from real algebraic geometry and polynomial homotopy continuation. We prove Theorem \ref{theorem:epsilon-local-rigidity} which supports Algorithm \ref{algorithm:epsilon-local-rigidity} for testing the $\varepsilon$-local rigidity of a framework. To show a framework in configuration $p_0 \in \mathbb{R}^{nd}$ is $\varepsilon$-locally rigid is to show that any continuous deformations of $p_0$ remain nearby to $p_0$ in configuration space. In particular, they do not exit a sphere of radius $\varepsilon$ about $p_0$ in configuration space. An $\varepsilon$-locally rigid framework may be locally rigid or flexible (see Figure \ref{fig:three-sphere-with-text}), but since $\varepsilon$ can be freely chosen by the user, this makes $\varepsilon$-local rigidity a useful property in any application where very small movements can be ignored. Another advantage is that $\varepsilon$-local rigidity applies to any configuration $p_0$, be it smooth or singular, generic or non-generic. Finally, we present Algorithms \ref{algorithm:discrete-flex} and \ref{algorithm:infinitesimal-flex-parameter-homotopy} which also use polynomial homotopy continuation to provide useful visual information for the scientist, even in examples that are currently too large for Algorithm \ref{algorithm:epsilon-local-rigidity}.

\bibliographystyle{plain}
\bibliography{references}

\end{document}